\documentclass[times,final, nopreprintline]{elsarticle}
\usepackage[left=3cm, right=3cm]{geometry}
\usepackage{graphicx}
\usepackage{placeins}
\usepackage{framed,multirow}

\usepackage{amssymb}
\usepackage{latexsym}
\usepackage{amsmath,amssymb,amsfonts,bm,amsthm,amscd}
\newtheorem{definition}{Definition}[section]
\newtheorem{proposition}{Proposition}[section]
\newtheorem{theorem}{Theorem}[section]
\newtheorem{lemma}{Lemma}[section]
\usepackage[ruled,linesnumbered,lined,boxed,commentsnumbered]{algorithm2e}
\usepackage{url}
\usepackage{xcolor}
\definecolor{newcolor}{rgb}{.8,.349,.1}

\begin{document}


\begin{frontmatter}

\title{Learning Hamiltonian systems with Pseudo-symplectic neural network }

\author[1]{Xupeng Cheng}
\author[1]{Lijin Wang}
\author[2]{Yanzhao Cao}
\author[1]{Chen Chen}

\address[1]{School of Mathematical Sciences, University of Chinese Academy of Sciences, Yu Quan Lu 19 (Jia), Beijing 100049, China}
\address[2]{Department of Mathematics and Statistics, Auburn University, Auburn, USA}


\begin{abstract}
Hamiltonian systems are fundamental and crucial  in both classical and quantum mechanics. Extracting governing differential equations of Hamiltonian systems from observational data can enhance understanding complex systems, predicting their dynamical behavior, and compensate the limitations of traditional modeling relying on first principles. 
Recent research on data-based neural network detection of Hamiltonian systems has focused on structure-preserving learning methods aiming to incorporate the intrinsic properties such as the symplecticity into the neural network architecture. Symplectic integrators have been used to generate the forward propagation to endow the neural network with symplectic structure. However this was mostly applied to separable Hamiltonian systems because symplectic integrators are in general implicit. Rescues for non-separable situations such as
symplectic networks (SympNets), non-separable symplectic neural networks (NSSNNs), generating function neural networks (GFNNs) etc. either use implicit algorithms or explicit algorithms with high dimensional neural networks in both widith and depths, making these algorithms the computationally intensive.  
In this paper we propose a novel one layer explicit  Pseudo-Symplectic Neural Network (PSNN) for learning non-separable  Hamiltonian systems.  This yields explicit neural network mapping with nearly exact symplecticity. To enhance the accuracy neural network simulation and compensate the loss of accuracy due to non-exact symplectcity, we design the learnable Pad\'e-type activation functions for neural network output. Our numerical experiments demonstrate that Pad\'e activation function outperform the classical ReLU activation function,  the Taylor activation function, and  activation function PAU. 
Our numerical experiments also demonstrate that our method is versatile in learning and forecasting a wide array of Hamiltonian systems, from the two-dimensional modified pendulum to the four-dimensional galactic dynamical system, surpassing conventional or cutting-edge benchmark neural networks in terms of prediction accuracy, long-term stability and energy-preservation, with shorter training period, smaller sample size and fewer network parameters. 
\end{abstract}


\end{frontmatter}



\section{Introduction}
Hamiltonian systems form a crucial class of dynamical systems with applications spanning various fields, from classical to quantum mechanics. Any motion with negligible dissipation  can be described by Hamiltonian systems (\cite{fengqinbook}), including phenomena such as the motion of celestial bodies, molecules, ideal pendulums, etc. A fundamental characteristic of these systems is the preservation of the symplectic structure in their phase spaces, which geometrically corresponds to the conservation of area along the evolution of the phase flow. Numerical discretization methods that retain this property are known as symplectic methods. Both theoretical analysis and empirical studies have demonstrated that these methods significantly outperform general numerical approaches in solving Hamiltonian systems, particularly with respect to stability and long-term prediction (\cite{fengqinbook,Hairer}).

With the development of machine learning, neural networks are becoming powerful tools for extracting governing differential equations from discrete time series of observational data of the evolution trajectories. ODE-nets (\cite{h-1, h-3, h-6}), PDE-nets (\cite{long1, long2,pinn}), and SDE-nets (\cite{bib-r2, bib-r5,bib-r7,bib-r4,bib-r66}) have been established to deal with the learning of ordinary, partial and stochastic differential equations.



Recently, learning Hamiltonian systems has drawn more and more attention.  In particular,  Hamiltonian Neural Network (HNN) (\cite{h-2})  pioneers the study on structure-preserving learning of differential equations. It proposed to learn the Hamiltonian systems by neural networks that respect the systems' physical conservation law such as the preservation of the energy (the Hamiltonian function), and  realized the Hamiltonian-preserving  neural network by directly parameterizing the Hamiltonian function instead of the total vector field as was done with regular neural ODEs, and trained the feedforward network using observations of both the state $\boldsymbol{y}$ of the Hamiltonian system and its time derivative $\dot{\boldsymbol{y}}$. 
Since then, numerous studies have explored neural networks capable of preserving the symplectic structure of underlying Hamiltonian systems while learning their governing equations. Many of these approaches build upon the fundamental idea of Hamiltonian Neural Networks (HNN), where the Hamiltonian function is parameterized using neural networks. This parameterization enables symplectic training, testing, and prediction by incorporating symplectic numerical methods into the network’s loss function or forward propagation, thereby ensuring a neural network with an inherent symplectic structure.
In \cite{h-4}, Zhu et al.  showed that applying symplectic integrators to generating the dynamical propagation of HNN can guarantee the existence of network target and improve the accuracy and generalization ability of HNN. In \cite{h-8}, David et al. established the Symplectic Hamiltonian Neural Networks (SHNNs), which apply symplectic integrators in the loss function to train a modified Hamiltonian and then corrected it to derive the true Hamiltonian with high accuracy using their relationship described by the Hamilton-Jacobi theory. Chen et al. (\cite{h-3}) proposed the use of multi-step symplectic leapfrog method in HNN to form their Symplectic Recurrent Neural Networks (SRNNs) for learning separable Hamiltonian systems.

There are also symplectic neural networks that are not based on HNNs, such as the reversible SympNet proposed by Jin et al. (\cite{h-5}). SympNet constructs complex symplectic networks by stacking simple symplectic modules, and it can approximate any symplectic mapping without computing gradients or using numerical integrators. However, achieving sufficient expressivity typically requires a sufficiently deep network composed of multiple simple symplectic modules. Additionally, the activation function must satisfy the 
$r$-finite condition (\cite{h-5}) to ensure universal approximation capability. Another approach for learning non-separable Hamiltonian systems is the Nonseparable Symplectic Neural Network (NSSNN), introduced by Xiong et al. (\cite{h-7}). NSSNN is built upon the explicit symplectic methods for non-separable Hamiltonian systems presented in \cite{Tao}. The explicitness of the symplectic scheme for non-separable Hamiltonian systems is realized by augmenting the original Hamiltonian system to be of double dimension through introducing auxiliary independent variables. This enables neural networks with symplectic structure even for non-separable Hamiltonian systems, but at the cost of the enlargement of dimension, as well as the inevitable calculations caused by composing sub-networks.  In \cite{tao}, Chen and Tao introduced the Generating Function Neural Network (GFNN) to learn exact symplectic mappings by parameterizing the generating function rather than the Hamiltonian function as in HNNs. This method ensures exact symplecticity without relying on numerical integrators, allowing it to learn both separable and non-separable Hamiltonian systems while preserving symplecticity. However, GFNN is implicit requiring fixed-point iterations for forward propagation, which adds computational overhead.

The existing constructions of symplectic neural networks for detecting and predicting Hamiltonian systems—whether as symplectic integrators or mappings—often rely on implicit integrators, iterative prediction steps, or increasing the network’s dimension through dimension augmentation or sub-network composition to enforce explicit symplectic mappings. However, these approaches significantly increase the computational complexity of the neural network. 

To reduce the computational complexity,Tong et al. \cite{h-6} proposed a novel symplectic neural network called   Taylor-net.   In this model, the gradient of the Hamiltonian function is parameterized by two sub-networks, each with a symmetric Jacobian, leveraging Taylor activations and Taylor series expansion. A fourth-order symplectic Runge-Kutta integrator is then employed as the overall dynamical propagator. However, the Taylor-net is limited to separable Hamiltonian systems.

In this paper, we propose a  Pseudo-Symplectic Neural Network (PSNN) for learning non-separable Hamiltonian systems. In PSNN, we  utilize pseudo-symplectic integrators \cite{Aubry} instead of traditional symplectic ones. These integrators preserve symplecticity to a high degree while remaining explicit mappings, enabling efficient and lightweight neural networks. 
To mitigate the loss of accuracy introduced by pseudo-symplecticity, we propose learnable Padé-type activation functions, inspired by Padé approximation theory, as a replacement for the Taylor activations used in \cite{h-6}. Unlike the existing Padé Activation Unit (PAU, \cite{pau}), our approach is specifically designed to enhance the performance of pseudo-symplectic neural networks (PSNNs) and make  the structure-preservation error negligible. Both theoretical analysis and numerical experiments demonstrate that integrating the pseudo-symplectic integrator with Padé-type activations significantly improves the learning of Hamiltonian systems, ranging from non-separable to separable and from low- to high-dimensional cases. Compared to benchmark neural networks employing conventional or state-of-the-art activation functions including ODE-net, HNN, and Taylor-net, as well as PSNN with ReLU, PAU and Taylor activations, our method achieves superior performance in terms of prediction accuracy, robustness, and efficiency, requiring fewer training samples, network parameters, and epochs.

The rest of the paper is arranged as follows. Section \ref{sec1} introduces some mathematical preliminaries and related works of the paper such as about the Hamiltonian systems and pseudo-symplectic methods, as well as the Taylor-nets, HNN and PAU. The learning framework of our PSNN, including the definition and analysis of the Pad\'e-type activation functions, and the pseudo-symplectic network architecture are presented in Section \ref{sec3}. Section \ref{tests} are numerical experiments on several Hamiltonian systems with non-separable or separable Hamiltonian functions, comparing the behavior of our PSNN using Pad\'e-type activation function with conventional or cutting-edge neural networks and activation functions. Section \ref{sec5} is a brief conclusion.

\section{Preliminaries}
\label{sec1}
\subsection{Hamiltonian systems and pseudo-symplectic methods}
A Hamiltonian system has the following form
\begin{align}\label{1}
\left\{\begin{array}{l}
\frac{d \boldsymbol{q}}{d t}=\frac{\partial H}{\partial \boldsymbol{p}},\quad \boldsymbol{p}(0)=\boldsymbol{p}_0, \\
\frac{d\boldsymbol{p}}{d t}=-\frac{\partial H}{\partial \boldsymbol{q}},\quad \boldsymbol{q}(0)=\boldsymbol{q}_0,
\end{array}\right.
\end{align}
where $t\ge 0$, $\boldsymbol{p}(t), \, \boldsymbol{q}(t)\in \mathbb{R}^d$ represents the generalized momentum and position respectively, and $H(\boldsymbol{p},\boldsymbol{q})$ is a smooth scalar function called the Hamiltonian function. Denote $\boldsymbol{y}=(\boldsymbol{p}^T,\boldsymbol{q}^T)^T$, the Hamiltonian system \eqref{1} can also be represented as
\begin{align}\label{2}
\frac{d\boldsymbol{y}}{dt}= J^{-1}\nabla H(\boldsymbol{y}),\quad \boldsymbol{y}(0)=\boldsymbol{y}_0,
\end{align}
where $\boldsymbol{y}_0=(\boldsymbol{p}_0^T,\boldsymbol{q}_0^T)^T$, $J=\left(\begin{array}{cc}
0 & I_d \\
-I_d & 0
\end{array}\right)$ with $I_d$ being the $d$-dimensional identity matrix. 

It has been proved that for any $\boldsymbol{y}_0\in \mathbb{R}^{2d}$ and $t\ge 0$, the flow $\Psi(t, \boldsymbol{y}_0): \boldsymbol{y}_0\rightarrow \boldsymbol{y}(t)$ of the Hamiltonian system \eqref{2} preserves the symplectic structure (see e.g. \cite{fengqinbook, feng1984}), i.e., 
\begin{align}\label{3}
\frac{\partial\Psi(t,\boldsymbol{y}_0)}{\partial \boldsymbol{y}_0}^T J\frac{\partial\Psi(t,\boldsymbol{y}_0)}{\partial \boldsymbol{y}_0}=J,
\end{align}
which can be geometrically interpreted as the conservation of area along the phase flow in the phase space. Symplecticity is the essential intrinsic property of Hamiltonian systems. A one-step numerical discretization method of \eqref{2} with time-step $h$, denoetd by $\Phi(h,\boldsymbol{y}_n):\boldsymbol{y}_n\rightarrow \boldsymbol{y}_{n+1}$ $(n=0,1,\cdots)$, is called a symplectic method, if it can preserve the symplecticity, namely if it satisfies
\begin{align}\label{3}
\frac{\partial \Phi(h,\boldsymbol{y}_n)}{\partial \boldsymbol{y}_n}^T J \frac{\partial \Phi(h,\boldsymbol{y}_n)}{\partial \boldsymbol{y}_n}=J, \quad \forall n=0, 1,\cdots.
\end{align}
Inheriting the geometric structure of the original systems, symplectic methods are proved theoretically and empirically to have profound numerical behavior with respect to accuracy and stability, superior to general-purpose methods, especially in long-term simulation (see e.g. \cite{Hairer}). 

However, symplectic methods are typically implicit, except for the cases when they are applied to linear or certain separable Hamiltonian systems. The implementation of implicit integrators is usually complex and computationally intensive. To get explicit symplecticity-preserving integrators for non-linear and non-separable systems, one can explore the application of the pseudo-symplectic methods which were proposed in {\cite{Aubry}}, and defined as follows:
\begin{definition}\label{d-1} {\cite{Aubry}}
A one-step numerical method $\Phi(h,\boldsymbol{y}_n)$ with step-size $h$ applied to the
Hamiltonian system \eqref{2} is called pseudo-symplectic of pseudo-symplectic order $s$, if it  satisfies
\begin{align}\label{pseud0}
\frac{\partial\Phi(h,\boldsymbol{y}_n)}{\partial \boldsymbol{y}_n}^TJ\frac{\partial\Phi(h,\boldsymbol{y}_n)}{\partial \boldsymbol{y}_n}=J+\mathcal{O}(h^{s+1}), \quad \forall n=0,1,\cdots.
\end{align}
\end{definition}
Although these methods lose certain (usually small) degree  of accuracy in symplecticity-preservation, they can be explicit and possess high convergence order, as demonstrated by the following explicit pseudo-symplectic Runge-Kutta method of convergence order 4 and pseudo-symplectic order 8 \cite{Misha}:
\begin{align}\label{RK}
\boldsymbol{\xi}_1&=\boldsymbol{y}_n,\quad \boldsymbol{\xi}_i = \boldsymbol{y}_n + h\sum_{j=1}^{i-1}a_{ij}\boldsymbol{F}_j \,\,\,\,(i=2,\cdots,7), \quad
\boldsymbol{F}_j=J^{-1}\nabla H(\boldsymbol{\xi}_j)\,\,(j=1,\cdots,7),\nonumber\\
\boldsymbol{y}_{n+1} &= \boldsymbol{y}_n + h\sum_{j=1}^{7}b_j\boldsymbol{F}_j,
\end{align}
where the parameters of the Runge-Kutta method are as follows with $\gamma=\frac{1}{4(2-\sqrt[3]{2})}$:
\begin{table}[h]
\centering
\begin{tabular}{c|ccccccc}
0\\
2$\gamma$&2$\gamma$\\
4$\gamma$& 0&4$\gamma$\\
1/2&2$\gamma$&0&1/2-2$\gamma$\\
1-4$\gamma$&0&4$\gamma$&0&1-8$\gamma$\\
1-2$\gamma$&2$\gamma$&0&1/2-2$\gamma$&0&1/2-2$\gamma$\\
1&0&4$\gamma$&0&1-8$\gamma$&0&4$\gamma$\\
\hline
&$\gamma$&2$\gamma$&1/4-$\gamma$&1/2-4$\gamma$&1/4-$\gamma$&2$\gamma$&$\gamma$\\
\end{tabular}
\caption{Parameters of the Runge-Kutta method \eqref{RK}.}
\label{tab:my_label0}
\end{table}


\subsection{Related work}
\subsubsection{The Taylor-nets}
Given observational phase trajectory data of the Hamiltonian system \eqref{2}:
\begin{align}\label{dataH0}
\mathcal{D}=\{(\boldsymbol{y}_0^i,y_1^i)|i=1,\cdots,N\},
\end{align}
neural network inversion of the governing ODE can be typically accomplished by the ODE-nets, which parameterize the latent vector field function $f(\boldsymbol{y}):=J^{-1}\nabla H(y)$ by $f_{net}(\boldsymbol{y},\theta)$ and incorporate it into a certain numerical integrator (such as the Euler method with time step $h$) to form the loss at time $t_1=t_0+h$ corresponding to the observation time of $\boldsymbol{y}_1^i$, i.e. 
$$Loss=\frac{1}{N}\sum_{i=1}^N\|\hat{\boldsymbol{y}}_1^i-\boldsymbol{y}_1^i\|^2$$
with 
$$\hat{\boldsymbol{y}}_1^i=y_1^i+hf_{net}(\boldsymbol{y}_1^i,\theta).$$
The ODE-nets are usually fully-connected feedforward neural networks, which are trained by using optimizers with backpropagation or its improved variants (\cite{h-1}).   

The HNN method (\cite{h-2}) begins the structure-preserving learning of Hamiltonian systems, which is committed to not only get the governing ODE but also capture its special structure. The HNN realized the structure recognition (energy preservation) by parameterizing the Hamiltonian function $H(\boldsymbol{y})$ of the Hamiltonian system, instead of learning the entire vector field function $f(\boldsymbol{y})$ as was done with the general ODE-nets, since the output $f_{net}(\boldsymbol{y},\theta^*)$ of general ODE-nets may not be the gradient of a certain function after being multiplied by $J$ from the left, owing to the training errors, which would consequently lead to erasure of the symplectic gradient flow structure of the original system. The loss function used by HNN is 
$$ Loss_{HNN}=\left\|\frac{\partial H_{net}}{\partial \boldsymbol{p}}-\frac{d \boldsymbol{q}}{d t}\right\|_2+\left\|\frac{\partial H_{net}}{\partial \boldsymbol{q}}+\frac{d \boldsymbol{p}}{d t}\right\|_2,$$
which needs observations of $(\dot{\boldsymbol{p}},\dot{\boldsymbol{q}})$. By learning the Hamiltonian function $H(\boldsymbol{y})$, the HNN acquires the symplectic structure of the system and thus enables more accurate prediction, energy preservation and faster training.

Taylor-nets (\cite{h-6}) introduce a new way of learning separable Hamiltonian systems, where the Hamiltonian function can be split into a function of $\boldsymbol{p}$ and a function of $\boldsymbol{q}$, namely,
\begin{align}
H(\boldsymbol{p},\boldsymbol{q})=T(\boldsymbol{p})+U(\boldsymbol{q}).
\end{align}
Taylor-nets inherited the structure-preserving idea, while directly learning the gradient of the Hamiltonian function $H(\boldsymbol{y})$, i.e. $T_{\boldsymbol{p}} 
\,\,(=\frac{\partial H}{\partial \boldsymbol{p}})$ and $U_{\boldsymbol{q}} 
\,\,(=\frac{\partial H}{\partial \boldsymbol{q}})$, instead of learning $H(\boldsymbol{y})$. This however poses a challenge to their network construction, since as gradients, the networks of them,  $T_{\boldsymbol{p}}(\boldsymbol{p},\theta_{\boldsymbol{p}})$ and $U_{\boldsymbol{q}}(\boldsymbol{q},\theta_{\boldsymbol{q}})$ should have symmetric Jacobian matrices. To realize this, they proposed the following form of networks:  
\begin{align}\label{8}
\boldsymbol{T}_{\boldsymbol{p}}\left(\boldsymbol{p}, \boldsymbol{\theta}_{\boldsymbol{p}}\right)=\left(\sum_{i=1}^S \boldsymbol{A}_i^T \circ f_i \circ \boldsymbol{A}_i-\boldsymbol{B}_i^T \circ f_i \circ \boldsymbol{B}_i\right) \circ \boldsymbol{p}+\boldsymbol{b}_1,
\end{align}
\begin{align}\label{9}
\boldsymbol{U}_{\boldsymbol{q}}\left(\boldsymbol{q}, \boldsymbol{\theta}_{\boldsymbol{q}}\right)=\left(\sum_{i=1}^S \boldsymbol{C}_i^T \circ f_i \circ \boldsymbol{C}_i-\boldsymbol{D}_i^T \circ f_i \circ \boldsymbol{D}_i\right) \circ \boldsymbol{q}+\boldsymbol{b}_2,
\end{align}
where 'o' denotes the function composition, $A_i, B_i,C_i,D_i$ are matrices of learnable network parameters with $d$ columns, $b_1,b_2$ are $d$-dimensional bias vectors of learnable network parameters, and the activation functions $f_i$ are the terms in the Taylor expansion:
\begin{align}
f_i(x)=\frac{1}{i!}x^i,
\end{align}
which acts element-wise on a $d$-dimensional vector, and are called Taylor activations. It is not difficult to check that the Jacobian matrices of \eqref{8} and \eqref{9}, i.e. $T_{\boldsymbol{pp}}(\boldsymbol{p},\theta_{\boldsymbol{p}})$ and $U_{\boldsymbol{qq}}(\boldsymbol{q},\theta_{\boldsymbol{q}})$ are both symmetric. Moreover, the use of the Taylor expansion form enhances the expressive capability of the neural network (\cite{h-6}).

For structure-preservation the Taylor-nets embedded the nets \eqref{8} and \eqref{9} into a symplectic Runge-Kutta numerical scheme which is explicit for separable Hamiltonian systems, to forward the inputs by fully-connected feedforward neural network with symplectic structure. 

The symplecticity-preserving network architecture together with the Taylor activation functions make the Taylor-nets significantly superior to other classical neural networks in learning Hamiltonian systems, in terms of prediction accuracy, convergence rate and robustness, with only small sample size and training data, as well as short training period.   

\subsubsection{The Pad\'e activation unit (PAU) }
The PAU (\cite{pau}) was proposed considering the excellent performance of the Pad\'e approximations to most existing activation functions,  and the merits of making activation functions learnable. Concretely, the PAUs are activation functions of the following form
\begin{align}\label{pau1}
F(x)=\frac{P(x)}{Q(x)}=\frac{\sum_{j=0}^m a_j x^j}{1+\left|\sum_{k=1}^n b_k x^k\right|}=\frac{a_0+a_1 x+a_2 x^2+\cdots+a_m x^m}{1+\left|b_1 x+b_2 x^2+\cdots+b_n x^n\right|},
\end{align}
where $a_i$ and $b_i$ are learnable parameters, which can be trained together with other network parameters, and bring flexibility to the activation function. The absolute value in the denominator is designed to avoid zero value of the denominator. 

The application of PAUs was shown to be able to increase the predictive performance and robustness of the neural networks.  

\section{Our method}\label{sec3}
\subsection{The learning framework}\label{sec31}
Given observational phase trajectory data \eqref{dataH0} of the Hamiltonian system \eqref{2}:
\begin{align}\label{dataH}
\mathcal{D}=\{(\boldsymbol{y}_0^i,y_1^i)|i=1,\cdots,N\},
\end{align}
where we assume that the time interval between each pair of observations $\boldsymbol{y}_0^i$ and $\boldsymbol{y}_1^i$ is $T$. Our object is to derive the ODE of the Hamiltonian system \eqref{2} from the observational data set $\mathcal{D}$ via neural network that captures the symplectic structure of the system. 

Unlike the Taylor-nets, we do not need to assume that $H(\boldsymbol{y})$ is separable. To preserve the symplectic gradient flow structure, we parameterize the gradient of the entire Hamiltonian function, i.e. $H_{\boldsymbol{y}}:=\frac{\partial H}{\partial \boldsymbol{y}}$ by neural network $H_{\boldsymbol{y,net}}(\boldsymbol{y},\theta)$ with learnable parameters $\theta$.
Given the input $\boldsymbol{y}_0^i$, the initial value of the Hamiltonian system at time $t_0$, denote the output of our network by $\hat{\boldsymbol{y}}_1^i$, which is the network approximation to the state of the Hamiltonian system at time $t_1=t_0+T$, i.e., $\boldsymbol{y}_1^i$. Then it should hold according to \eqref{2} that
\begin{align}\label{haty1}
\hat{\boldsymbol{y}}_1^i = \boldsymbol{y}_0^i + \int_0^TJ^{-1}H_{\boldsymbol{y},net}(\boldsymbol{y}(t),\theta)dt.
\end{align}
To implement the right-hand side of \eqref{haty1}, we set up the Pseudo-Symplectic Neural Network (PSNN), which is a fully-connected feedforward neural network processed by iteratively using (K times) an explicit pseudo-symplectic numerical integrator $\Phi$ with time-step $h$ satisfying $h=T/K$. Then the output $\hat{\boldsymbol{y}}_1^i$ of the PSNN can be expressed as: 
\begin{align}\label{14}
\hat{\boldsymbol{y}}_1^i = \Phi^K(h, \boldsymbol{y}_0^i, H_{\boldsymbol{y},net}),
\end{align}
where 
\begin{align*}
\Phi^K:={\underbrace{\Phi\circ \Phi\circ\cdots\circ\Phi}_{\mbox{$K$ times}}},
\end{align*}
and the schematic diagram of the PSNN is demonstrated in Figure \ref{psnn}.

By Theorem \ref{thmpseu} in section \ref{psnnstr}, $\Phi ^ K$ is still pseudo-symplectic for our chosen pseudo-symplectic integrator $\Phi$, which indicates that the PSNN defines a pseudo-symplectic transmission of input data. The reason for selecting the pseudo-symplectic scheme instead of the symplectic ones is due to the usual implicitness of the latter when applied to non-separable non-linear Hamiltonian systems, which hinders fluent transmission of input signals in the network, while the pseudo-symplectic scheme can be explicit even for non-separable non-linear systems. This makes the PSNN capable of simulating the flow of all Hamiltonian systems instead of only separable ones, with only small bias of symplecticity that produces almost negligible effect in empirical analysis.

Based on \eqref{14}, we define the loss function of the PSNN to be the mean-square discrepancy between $\hat{\boldsymbol{y}}_1^i$ and $\boldsymbol{y}_1^i$:
\begin{align}\label{loss1}
Loss_{PSNN} = \frac{1}{N}\sum_{i=1}^{N}\|\hat{\boldsymbol{y}}_1^i - \boldsymbol{y}_1^i\|_2^2.
\end{align}
Next we introduce our construction of $H_{\boldsymbol{y},net}$ and the selection of the integrator $\Phi$.
\subsection{Construction of $H_{y,net} $ with Pad\'e-type activation functions}
Considering the significant role of Pad\'e approximations in approximation theory, as well as the requirements on activation functions in neural networks, we propose the following Pad\'e-type activation functions:
\begin{definition}\label{pt0}
A Pad\'e-type activation function of degree $(L,M)$, denoted by $PT^{L,M}:\mathbb{R}\rightarrow\mathbb{R}$ is defined as
\begin{align}\label{ptf}
PT^{L,M}(x):=\frac{\sum_{j=0}^L c_j x^j}{d_M(x)},
\end{align}
where $L,M\in \mathbb{Z}_{+}$, $L\neq M$, $d_M(x)$ is a prescribed polynomial of degree $M$ which has no roots in $\mathbb{R}$, and the parameters $c_j$ $(j=0,\cdots,L)$ in the numerator are learnable. 
\end{definition}
As activation functions, the $PT^{L,M}$ should enable neural networks where they are embedded to satisfy the universal approximation theorem (\cite{Cybenko,Hornik,Pinkus}). Next we show this via referring to the following definition and proposition from \cite{Kidger}. 
\begin{definition}(\cite{Kidger})\label{d1}
Let $\rho: \mathbb{R} \rightarrow \mathbb{R}$ and $n, m, k \in \mathbb{N}$. Then let $\mathcal{N} \mathcal{N}_{n, m, k}^\rho$ represent the class of functions $\mathbb{R}^n \rightarrow \mathbb{R}^m$ described by feedforward neural networks with $n$ neurons in the input layer, $m$ neurons in the output layer, and an arbitrary number of hidden layers, each with $k$ neurons with activation function $\rho$. Every neuron in the output layer has the identity activation function.
\end{definition}

\begin{proposition}(\cite{Kidger})\label{pr1}
Let $\rho: \mathbb{R} \rightarrow \mathbb{R}$ be any continuous nonpolynomial function which is continuously differentiable at at least one point, with nonzero derivative at that point. Let $U \subset \mathbb{R}^n$ be compact. Then $\mathcal{N N}_{n, m, n+m+1}^\rho$ is dense in $C\left(U ; \mathbb{R}^m\right)$ with respect to the uniform norm.
\end{proposition}

From Proposition \ref{pr1} we get the following result.

\begin{theorem}\label{t1}
For any given $L,M\in\mathbb{Z}_{+}$ $(L\neq M)$, the class of feedforward neural networks  $\mathcal{N N}_{n, m, n+m+1}^{PT^{L,M}}$ with Pad\'e-type activation function $PT^{L,M}$ are dense in $C\left(U ; \mathbb{R}^m\right)$ with respect to the uniform norm, where $n,m\in \mathbb{N}$ and $U\subset{\mathbb{R}^{n}}$ is compact.
\end{theorem}
\begin{proof}
According to Proposition \ref{pr1}, we only need to check that for any given $L,M\in\mathbb{R}_{+}$ $(l\neq M)$, the function $PT^{L,M}(x)$ is continuously differentiable at at least one point, with nonzero derivative at that point. Denoting the numerator of $PT^{L,M}(x)$ by $N_L(x)$, we have
\begin{align}\label{deri}
\frac{d}{dx}PT^{L,M}(x)=\frac{N'_{L}(x)d_M(x)-N_{L}(x)d'_{M}(x)}{d_M(x)^2}.
\end{align}
Since $d_{M}(x)$ has no roots in $\mathbb{R}$, then $PT^{L,M}(x)$ is continuously differentiable everywhere in $\mathbb{R}$. On the other hand, due to $L\neq M$, it is true that $PT^{L,M}(x)$ is not a constant, which implies that $\frac{d}{dx}PT^{L,M}(x)\nequiv 0$. That is, there exists at least one point $x_0\in\mathbb{R}$ such that $\frac{d}{dx}PT^{L,M}(x_0)\neq 0$. This ends the proof.
\end{proof}

Similar to the Taylor-nets where gradients of functions are to be approximated by neural networks, our PSNN uses $H_{\boldsymbol{y},net}$ to approximate the gradient of the entire Hamiltonian function $H$ which can be nonseparable, i.e. $\frac{\partial H}{\partial \boldsymbol{y}}$. In this way the
$H_{\boldsymbol{y},net}$ needs to have symmetric Jacobian matrix. To achieve this, we follow the structure of $T_{\boldsymbol{p}}$ and $U_{\boldsymbol{q}}$ in the Taylor-nets given in \eqref{8}-\eqref{9} to build the network with one hidden layer for $H_{\boldsymbol{y},net}$ as:
\begin{align}\label{hynetform}
H_{\boldsymbol{y},net}(\boldsymbol{y},\theta)=\sum_{i=1}^S A_i^T\circ PT_i^{L,M} \circ A_i\circ \boldsymbol{y}-\sum_{i=1}^S B_i^T\circ PT_i^{L,M}\circ B_i\circ \boldsymbol{y}+b,
\end{align}
where 
\begin{align}\label{pti}
PT_i^{L,M}(x):=\frac{\sum_{j=0}^Lc_{ij}x^j}{d_M(x)}, \quad i=1,\cdots, S
\end{align}
are Pad\'e-type activation functions with $c_{ij}$ being learnable parameters for $i=1,\cdots,S$, $j=0,\cdots,L$. These functions act element-wise on the vector $A_i\circ \boldsymbol{y}$ to form a new vector with the same dimension of $A_i\circ\boldsymbol{y}$. The  matrices $A_i,\,\,B_i$ of learnable parameters are of dimension $l\times 2d$ with $l$ being the number of hidden neurons. The bias vector $b$ is of dimension $2d$. Thus the parameters of $H_{\boldsymbol{y},net}$ include the weights $A_i,\,B_i$, the bias $b$, as well as the learnable parameters $c_{ij}$ in the Pad\'e-type activation functions, namely
\begin{align}\label{thetaall}
\theta=\{A_i,\,B_i,\,b,\,c_{ij}|i=1,\cdots,S, j=0,\cdots,L\}.
\end{align}

The schematic diagram of the network $H_{\boldsymbol{y},net}$ is shown in Figure \ref{Hynet}.
\begin{figure}[h]
\centering
\includegraphics[width=\linewidth]{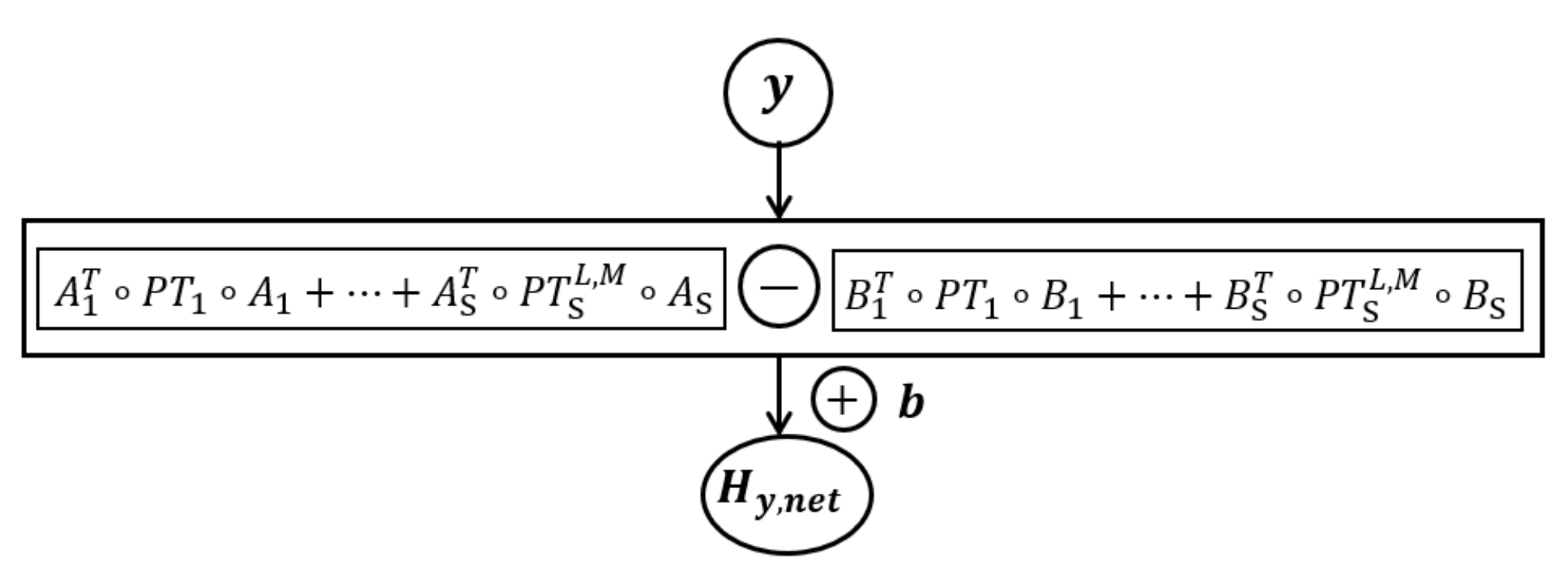}
\caption{Schematic diagram of $H_{\boldsymbol{y},net}$}\label{Hynet}
\end{figure}
The difference between our network simulation of the gradient and that of the Taylor-nets mainly lies in two aspects. One is the detachment from the separable assumption of $H$ which allows unified parameterization of the entire gradient of $H$ by $H_{\boldsymbol{y},net}$ rather than by two nets ($T_{\boldsymbol{p}}(\boldsymbol{p},\theta_{\boldsymbol{p}})$ and $U_{\boldsymbol{q}}(\boldsymbol{q},\theta_{\boldsymbol{q}})$) separately, which enlarges the scope of application to even nonseparable Hamiltonian systems.
The other is the use of the Pad\'e-type activation functions instead of the Taylor activation functions. Being learnable, the Pad\'e-type activations are shown by numerical experiments in section \ref{tests} to outperform the Taylor activation functions in learning accuracy even with less network parameters and hidden neurons.

Meanwhile, the empirical analysis shows that, compared to the conventional activation functions such as the ReLU, the Pad\'e-type activation functions allow for the use of sparse training data and fewer training iterations. In contrast to PAU where the denominator is equipped with an absolute value, our Pad\'e-type activation functions avoid zero denominators by prescribing a polynomial denominator $d_M(x)$ without real roots, which simplifies the training process with respect to the backpropagation calculations and parameters size. Moreover, the Pad\'e-type activation functions remain a more natural form of Pad\'e approximants which may imply better inheritance of their high performance in approximation, where the choice of $d_M(x)$ is also flexible to be adaptive to specific problems at hand.

In summary, combining the entire-gradient architecture with the application of learnable Pad\'e-type activation functions, our $H_{\boldsymbol{y},net}$ embedded into the highly structure-preserving PSNN can learn general (nonseparable) Hamiltonian functions with significantly improved precision and robustness, even with sparser data. 

\subsection{The pseudo-symplectic network structure}\label{psnnstr}
As described in section \ref{sec31}, we need an explicit pseudo-symplectic integrator $\Phi$ in our PSNN, which we choose as the pseudo-symplectic Runge-Kutta method given by \eqref{RK} with parameters listed in Table \ref{tab:my_label0}. This method is of convergence order $4$ and pseudo-symplectic order $9$. It will be iteratively applied $K$ times $(k\ge1)$ in the PSNN, so we need to show that $\Phi^K$ is still pseudo-symplectic.
\begin{lemma}\label{lemma1} Let $\|\cdot\|$ be any matrix norm, and $\Phi(h,\boldsymbol{y})$ be a one-step pseudo-symplectic numerical integrator for Hamiltonian system \eqref{2} with pseudo-symplectic order $s$ ($s\ge0$) and time-step $h$ satisfying  $h<1$. Then for any $\boldsymbol{y}$, $\left\|\frac{\partial\Phi(h,\boldsymbol{y})}{\partial \boldsymbol{y}}\right\|$ is bounded.

 \end{lemma}
\begin{proof}
Since all matrix norms are equivalent in finite dimensional matrix spaces, we only need to show that the result holds for matrix 2-norm.

Denote $A=\frac{\partial\Phi(h,\boldsymbol{y})}{\partial \boldsymbol{y}}$. According to the definition of pseudo-symplectic integrator \eqref{pseud0}, we have
\begin{align*}
\|A^T J A\|_2\le \|J\|_2+O(h^{s+1})=1+O(h^{s+1})<\infty,
\end{align*}
which then implies that $\|A^T J A\|_{m_{\infty}}$ is bounded such that all elements in $A^T J A$ are finite. Consequently, $\det (A^T J A)$ is finite indicating that $\det (A)<\infty$. On the other hand,
\begin{align*}
J+O(h^{s+1})=J(I-J\cdot O(h^{s+1})),
\end{align*}
where $$\|J\cdot O(h^{s+1})\|_2=\|J\|_2O(h^{s+1})=O(h^{s+1})<1$$
owing to $h<1$. Thus $J+O(h^{s+1})$ is invertible, meaning that $A$ is invertible. Consequently, $A^TA$ is an $2d\times 2d$- dimensional positive definite matrix with real positive eigen-values
$$\lambda_1\ge\lambda_2\ge\cdots\ge\lambda_{2d}>0.$$
Since
$$\lambda_1\cdot\lambda_2\cdot\dots\cdot\lambda_{2d}=\det(A^TA)=(\det(A))^2<\infty,$$
it holds that $0<\lambda_{2d}<\infty$. Then we have
$$0<\|(A^T)^{-1}\|_2^2=\lambda_{max}\left((A^TA)^{-1}\right)=\frac{1}{\lambda_{2d}}<\infty.$$
Therefore, 
\begin{align*}
\|A\|_2=\|J^{-1}(A^T)^{-1}A^T J A\|_2\le \|(A^T)^{-1}\|_2\cdot \|A^T J A\|_2<\infty.
\end{align*}
\end{proof}

\begin{theorem}\label{thmpseu} 
Let $\Phi(h,\boldsymbol{y})$ be a pseudo-symplectic integrator with pseudo-symplectic order $s$ ($s\ge0$) and time step $h$ $(h<1)$ for Hamiltonian system \eqref{2}. Then for any $K\in \mathbb{Z}_{+}$, $K$ times composition of $\Phi$, namely $\Phi^K$ is still pseudo-symplectic with pseudo-symplectic order $s$.
\end{theorem}
\begin{proof}
For brevity we omit writing the time step $h$ inside the pseudo-symplectic integrators in the following, e.g. $\Phi(\boldsymbol{y}):=\Phi(h,\boldsymbol{y})$. We prove the result by mathematical induction on $K$. The assertion holds naturally for $K=1$. Assume that it holds for $K\le l$ $(l\ge1)$, then for $K=l+1$ we have
\begin{align*}
\frac{\partial \Phi^{l+1}(\boldsymbol{y})}{\partial \boldsymbol{y}}=\frac{\partial \Phi (\Phi^l(\boldsymbol{y}))}{\partial \Phi^l (\boldsymbol{y})}\frac{\partial \Phi^l(\boldsymbol{y})}{\partial \boldsymbol{y}}.
\end{align*}
Thus
\begin{align}\label{before}
\frac{\partial \Phi^{l+1}(\boldsymbol{y})}{\partial \boldsymbol{y}}^T J \frac{\partial \Phi^{l+1}(\boldsymbol{y})}{\partial \boldsymbol{y}}=\frac{\partial \Phi^l(\boldsymbol{y})}{\partial \boldsymbol{y}}^T \frac{\partial \Phi (\Phi^l(\boldsymbol{y}))}{\partial \Phi^l(\boldsymbol{y})}^T J \frac{\partial \Phi (\Phi^l(\boldsymbol{y}))}{\partial \Phi^l (\boldsymbol{y})}\frac{\partial \Phi^l(\boldsymbol{y})}{\partial \boldsymbol{y}}.
\end{align}
Since $\Phi$ is pseudo-symplectic, we have
\begin{align}\label{middle}
\frac{\partial \Phi (\Phi^l(\boldsymbol{y}))}{\partial \Phi^l(\boldsymbol{y})}^T J \frac{\partial \Phi (\Phi^l(\boldsymbol{y}))}{\partial \Phi^l (\boldsymbol{y})}=J+O(h^{s+1}).
\end{align}
Substituting \eqref{middle} into \eqref{before} we obtain
\begin{align*}
\frac{\partial \Phi^{l+1}(\boldsymbol{y})}{\partial \boldsymbol{y}}^T J \frac{\partial \Phi^{l+1}(\boldsymbol{y})}{\partial \boldsymbol{y}}&=\frac{\partial \Phi^l(\boldsymbol{y})}{\partial \boldsymbol{y}}^T(J+O(h^{s+1})) \frac{\partial \Phi^l(\boldsymbol{y})}{\partial \boldsymbol{y}}=J+O(h^{s+1}),
\end{align*}
where the last equality holds because $\Phi^l$ is pseudo-symplectic by the assumption of induction, and $$\left\|\frac{\partial \Phi^l(\boldsymbol{y})}{\partial \boldsymbol{y}}^T(O(h^{s+1})) \frac{\partial \Phi^l(\boldsymbol{y})}{\partial \boldsymbol{y}}\right\|_2\le \left\|\frac{\partial \Phi^l(\boldsymbol{y})}{\partial \boldsymbol{y}}\right\|_2^2 O(h^{s+1})=O(h^{s+1})$$ by Lemma \ref{lemma1} which ensures that $\left\|\frac{\partial \Phi^l(\boldsymbol{y})}{\partial \boldsymbol{y}}\right\|_2^2$ is bounded.
This ends the proof.

\end{proof}
As illustrated by the schematic diagram in Figure \ref{psnn}, the PSNN generates a pseudo-symplectic transmission from the input to the output, according to the above Theorem \ref{thmpseu}. In contrast to the symplectic Taylor-nets for separable Hamiltonian systems, the pseudo-symplectic PSNN can be applied to general nonseparable Hamiltonian systems with close preservation of the symplectic structure that enables accurate long-term prediction, as is shown by numerical experiments in section \ref{tests}.
\begin{figure}[h]
\centering
\includegraphics[width=\linewidth]{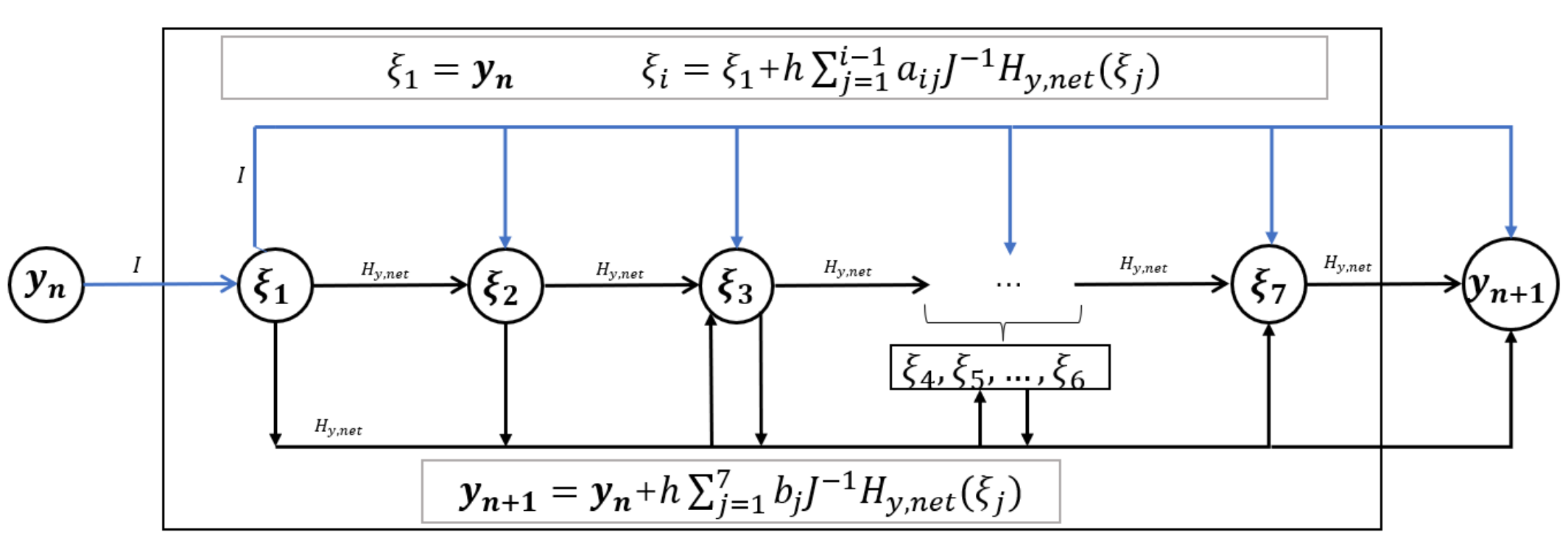}
\caption{Schematic diagram of PSNN, where $I$ represents the identity mapping.}\label{psnn}
\end{figure}
The overall algorithm of learning the gradient of Hamiltonian functions for Hamiltonian systems via the pseudo-symplectic neural network is summarized in the table of Algorithm \ref{alg1}.

\IncMargin{-1em}
\begin{algorithm}[H]
\SetKwInOut{Input}{Parameters}
\caption{Process of learning Hamiltonian systems via PSNN}\label{alg1}
\LinesNumbered 
\KwIn{Training data $\mathcal{D}_{train}=\{(\boldsymbol{y}_0^{(i)},\boldsymbol{y}_1^{(i)})_{train}|i=1,\cdots,N_{train}\}$, testing data $\mathcal{D}_{test}=\{(\boldsymbol{y}_0^{(i)},\boldsymbol{y}_1^{(i)})_{test}|i=1,\cdots,N_{test}\}$, time interval $T$ between $\boldsymbol{y}_0^{(i)}$ and $\boldsymbol{y}_1^{(i)}$, pseudo-symplectic integrator $\Phi$ with time-step $h$, number of Pad\'e-type activation units $S$, degree of Pad\'e-type activation functions $(L,M)$}
\Input{ parameters $\theta$ of $H_{\boldsymbol{y},net}(\cdot,\theta)$, learning rate $l_r$, number of training epochs $n_{epoch}$}

\KwOut{$H_{\boldsymbol{y},net}$}
\BlankLine
\For{$j=1:n_{epoch}$}{
\For{$i=1:N_{train}$}{
$\hat{y}_1^{(i)}=\Phi^{T/h}(\boldsymbol{y}_0^{(i)},h,H_{\boldsymbol{y},net}(\cdot,\theta))$\;}
$Loss_{PSNN} = \frac{1}{N_{train}}\sum_{i=1}^{N_{train}}\|\hat{\boldsymbol{y}}_1^{(i)} - \boldsymbol{y}_1^{(i)}\|_2^2 $\;
optimize $Loss_{PSNN}$ by $Adam$ and update $\theta$\;
}
\end{algorithm}
\DecMargin{-1em}

\section{Numerical experiments}\label{tests}
The first three examples apply the PSNN with Pad\'e-type activation function to learning non-separable Hamiltonian systems, where the Taylor-nets are not applicable. Therefore we compare it with the classical HNN and ODE-net (\cite{h-1}) in these examples. Meanwhile, PSNN using Pad\'e-type activation function is compared with PSNN using Taylor, PAU and ReLU activation functions, which refer to replacing the Pad\'e-type activation function in $H_{\boldsymbol{y},net}$ \eqref{hynetform} by Taylor, PAU and ReLU, respectively, i.e.:
\begin{align}
H^{Taylor}_{\boldsymbol{y},net}(\boldsymbol{y},\theta)=\sum_{i=1}^S A_i^T\circ f_i \circ A_i\circ \boldsymbol{y}-\sum_{i=1}^S B_i^T\circ f_i\circ B_i\circ \boldsymbol{y}+b,
\end{align}
where $f_i(x)=\frac{1}{i!}x^i$,
\begin{align}
H^{ReLU}_{\boldsymbol{y},net}(\boldsymbol{y},\theta)=\sum_{i=1}^S A_i^T\circ ReLU \circ A_i\circ \boldsymbol{y}-\sum_{i=1}^S B_i^T\circ ReLU\circ B_i\circ \boldsymbol{y}+b,
\end{align}
and 
\begin{align}
H^{PAU}_{\boldsymbol{y},net}(\boldsymbol{y},\theta)=\sum_{i=1}^S A_i^T\circ PAU_i \circ A_i\circ \boldsymbol{y}-\sum_{i=1}^S B_i^T\circ PAU_i\circ B_i\circ \boldsymbol{y}+b,
\end{align}
where $PAU_i$ represents the PAU activation function $\frac{\sum_{j=0}^m a_{ij} x^j}{1+\left|\sum_{k=1}^n b_{ik} x^k\right|}$, with $a_{ij}$, $b_{ik}$ being learnable parameters.

The fourth example of this section considers a separable Hamiltonian system where both the PSNN with Pad\'e-type activation function  and the Taylor-nets are applicable. There we compare the learning behavior of the two neural networks.

In all the experiments, we use 'Adam' as the optimization algorithm.
For the Pad\'e-type activation function, we select $d_M(x)=2+2x+x^2$ which has no roots in $\mathbb{R}$ to be the denominator. That is $M=2$. According to the Pad\'e approximation theory, the approximation is most effective when the degrees of the numerator and the denominator are equal or differ by $1$. Therefore, since we need $L\neq M$, we choose $L=3$ in our experiments. The number $S$ of Pad\'e-type summands in $H_{\boldsymbol{y},net}$ \eqref{hynetform} is set to $S=4$ based on our empirical analysis, but for the higher-dimensional problem in Example 3 we let $S=6$. The learnable parameters $c_{ij}$ in the Pad\'e-type activation functions \eqref{pti} are initialized as $c_{ij}=0,\,\,(i=1,\cdots,S, j=0,\cdots,L)$ in the training. As in \cite{h-6}, the network parameters $A_i$,$B_i$ and $b$ in (\ref{hynetform}) are initialized as
$A_i, B_i \sim \mathcal{N}\left(0, \sqrt{2 /2ld}\right)$ and $b\sim \mathcal{N}\left(0, 1\right)$, where $2d$ is the dimension of the Hamiltonian system and $l$ is the width of the hidden layer. 



\subsection{Example 1: Movement of a bead on a wire} In this example we consider learning a Hamiltonian system which describes the movement of a bead on a wire \cite{Arnol'd}, where the Hamiltonian function is
\begin{align}\label{h-1}
H(p,q)=\frac{p^2}{2(1+U^{'}(q)^2)}+U(q),
\end{align}
with $p,q$ being scalar and $U(q)=0.1q(q-1)$. 

For the training data $\mathcal{D}_{train}=\{(\boldsymbol{y}_0^{(i)},\boldsymbol{y}_1^{(i)})_{train}|i=1,\cdots,N_{train}\}$ we sample the initials $\{\boldsymbol{y}_0^{(i)}=(p_0^{(i)},q_0^{(i)})|i=1,\cdots,N_{train}\}$ uniformly from the region $V=[-2,2]\times[-2,2]$, and generate the $\{\boldsymbol{y}_1^{(i)}=(p_1^{(i)},q_1^{(i)})|i=1,\cdots, N_{train}\}$ via the midpoint rule with time step $h=0.01$, thereby also the time interval between $\boldsymbol{y}_0^{(i)}$ and $\boldsymbol{y}_1^{(i)}$ is taken as $T=0.01$, and the iteration number $K$ of the pseudo-symplectic integrator $\Phi$ is $K=T/h=1$ . Testing data $\mathcal{D}_{test}=\{(\boldsymbol{y}_0^{(i)},\boldsymbol{y}_1^{(i)})_{test}|i=1,\cdots,N_{test}\}$ are collected in the same way but independent from that for $\mathcal{D}_{train}$. In the experiment we take $N_{train}=15$ and $N_{test}=100$. The width of the PSNN is set to $l=32$, and the number of training epochs is $n_{epoch}=1500$. 

\begin{figure}[ht]
\centering
\includegraphics[width=\linewidth]{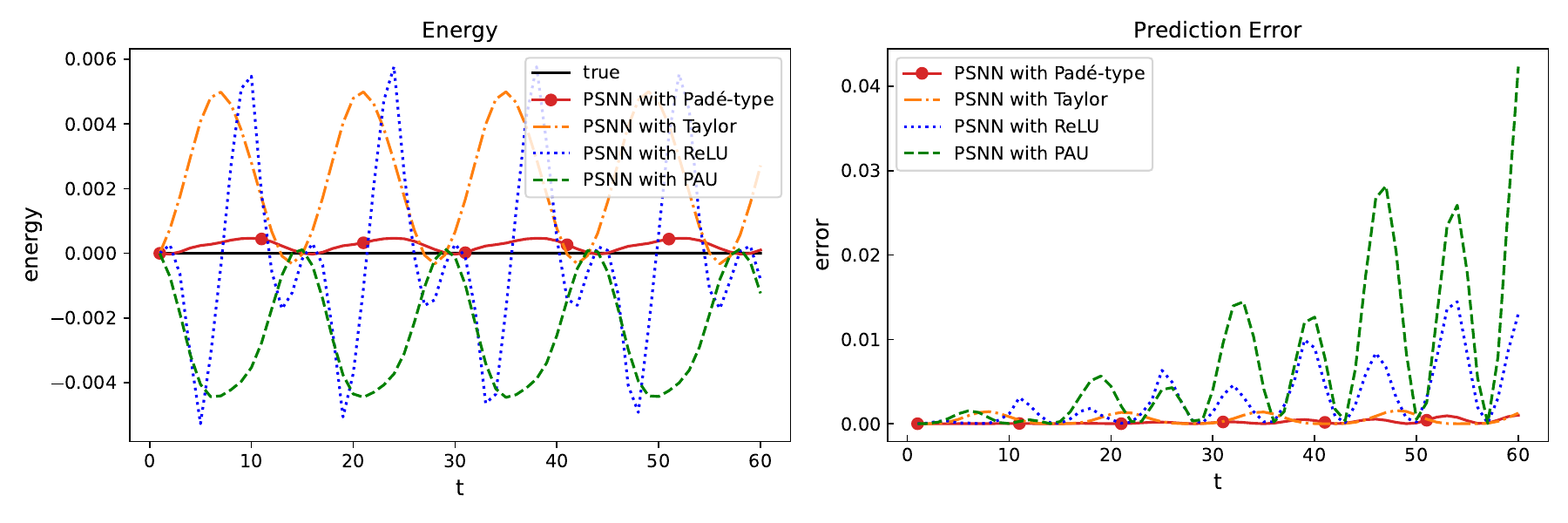}
\caption{Evolution of energy learned by PSNN with different activation functions (left); Prediction error by PSNN with different activation functions (right).}\label{2.3}
\end{figure}

It is well known that the Hamiltonian function $H(\boldsymbol{y}(t))$ of a Hamiltonian system is an invariant quantity along time evolution of the flow. It is often called the total energy of the system (\cite{Hairer}). There has been works attempting to create neural network learning of Hamiltonian systems that can preserve the energy, such as the HNN (\cite{h-2}). Figure \ref{2.3} (left panel) shows the evolution of $H(\boldsymbol{y}(t))$ \eqref{h-1} on $t\in[0,60]$ learned by the PSNN with four different activation functions, namely the Pad\'e-type, Taylor, PAU and ReLU. It is obvious that the Pad\'e-type activation function outperforms the other three significantly in energy preservation.

The right panel of Figure \ref{2.3} compares the prediction errors on $t\in[0,60]$ produced by the PSNN with different activation functions, which is defined as
\begin{align}\label{errorp-1}
error(t)=\|\hat{\boldsymbol{y}}(t)-\boldsymbol{y}(t)\|^2_2,
\end{align}
where $\boldsymbol{y}(t)$ represents the true solution of $\boldsymbol{y}$ at time $t$ which is simulated by the midpoint rule with tiny time step, and $\hat{\boldsymbol{y}}(t)$ denotes the approximation to $\boldsymbol{y}(t)$ via the PSNN. Obviously the Pad\'e-type activation function yields the least error, while the ReLU produces the largest error.
\begin{figure}[ht]
\centering
\includegraphics[width=\linewidth]{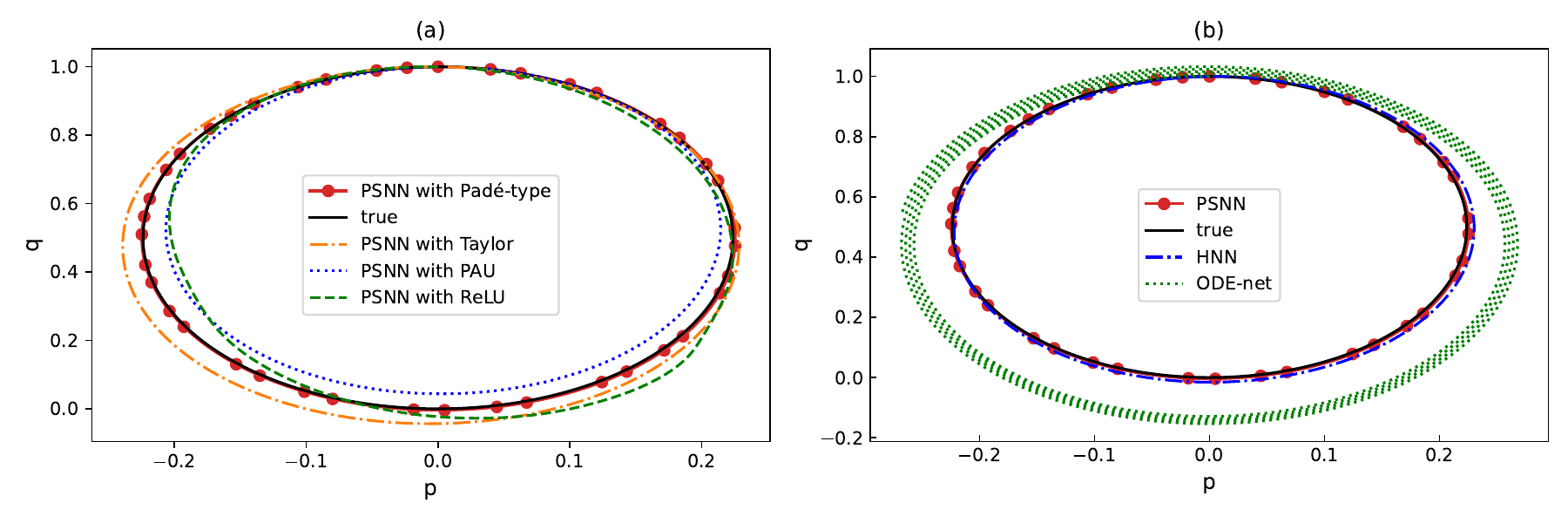}
\caption{(a) Predicted phase trajectory on $t\in[0,60]$ by PSNN with different activation functions; (b) Predicted phase trajectory on $t\in[0,60]$ by different neural networks.}\label{1.5}
\end{figure}

The left panel of Figure \ref{1.5} draws  predicted phase trajectories by using PSNN with the Pad\'e-type, Taylor, PAU and ReLU activation functions respectively, while the right panel illustrates predicted phase trajectories by PSNN with Pad\'e-type activation, HNN and ODE-net accordingly. In both panels the true solution is drawn as reference. Clearly, PSNN with Pad\'e-type activation creates the most accurate prediction, and the Pad\'e-type activation function performs the best among the four activation functions. Moreover, the training samples of HNN and ODE-net are 1000, which far exceeds the 15 training data required for PSNN with  Pad\'e-type activation function. It can be seen from the right panel that the structure-preserving neural networks, the PSNN and the HNN, exhibit good prediction behavior, while the non-structure-preserving ODE-net produces a certain thickness on the curve.

\subsection{Example 2: A modified pendulum problem}
Here we consider a 2-dimensional modified pendulum problem with Hamiltonian function (\cite{Hairer2})
\begin{align}\label{h-2}
H(p,q)=\frac{p^2}{2}-cos(q)(1-\frac{p}{6}).
\end{align}

The data and neural network setting are the same with those for Example 1. 

\begin{figure}[ht]
\centering
\includegraphics[width=\linewidth]{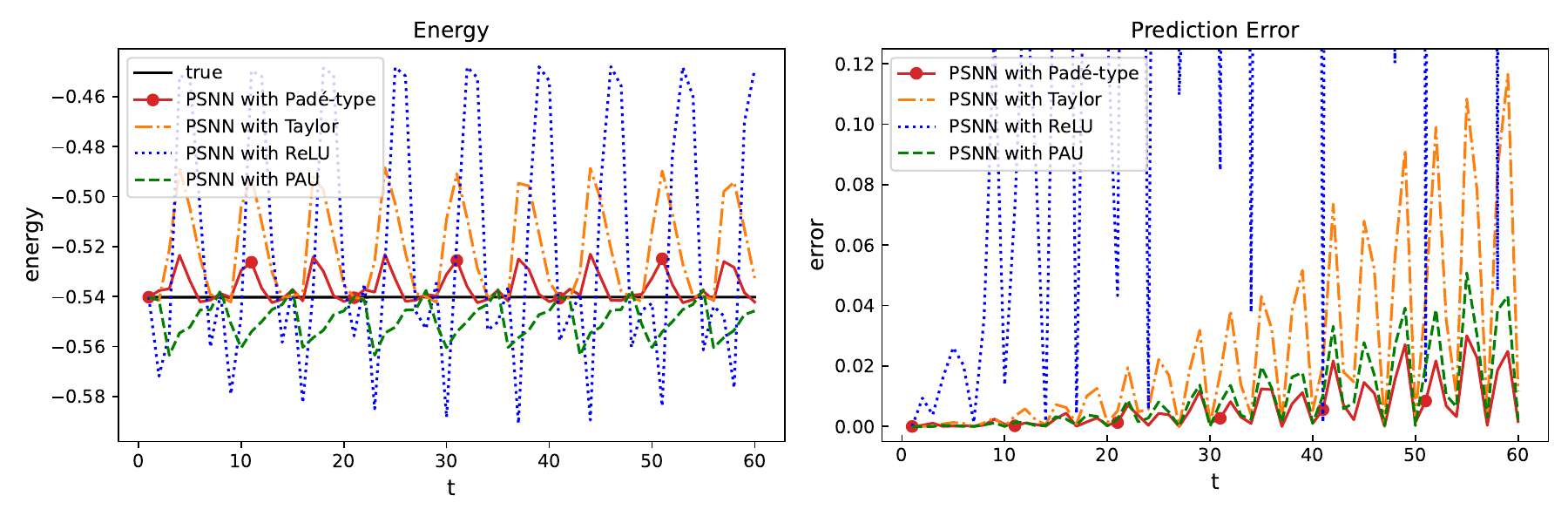}
\caption{Evolution of energy learned by PSNN with different activation functions (left); Prediction error by PSNN with different activation functions (right).}\label{1.1}
\end{figure}

The left panel of Figure \ref{1.1} shows evolution of the Hamiltonian function \eqref{h-2} on $t\in[0,60]$ learned by PSNN with different activation functions, where it is evident that the PSNN with Pad\'e-type activation function can better preserve the energy than the other three. The right panel compares the prediction error \eqref{errorp-1} on $t\in [0,60]$ produced by PSNN with different activation functions, where the Pad\'e-type activation function creates the most accurate prediction.

\begin{figure}[ht]
\centering
\includegraphics[width=\linewidth]{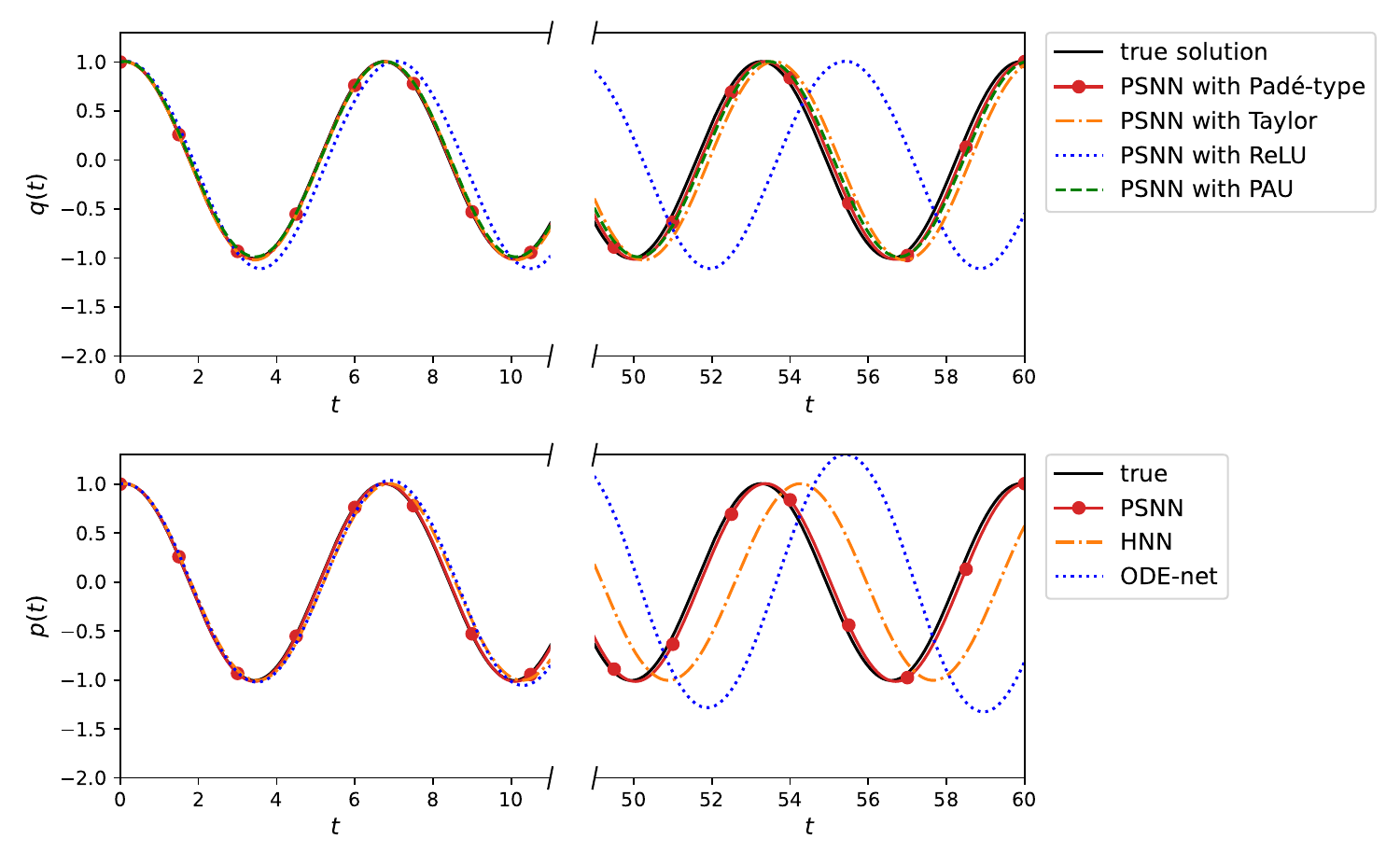}
\caption{Predicted $p(t)$ generated by PSNN with different activation functions (up); Predicted $p(t)$ generated by PSNN, HNN and ODE-net, where PSNN uses the Pad\'e-type activation function (low).}\label{2.5}
\end{figure}
In Figure \ref{2.5} we compare the prediction of $p(t)$ generated by our model and by others, where the upper panel contributes to the comparison of different activation functions, and the lower panel contrasts the prediction effect of applying PSNN with that by HNN and ODE-net. It can be seen that our PSNN model and the Pad\'e-type activation function provide the best long-term prediction among the comparators.
\begin{figure}[ht]
\centering
\includegraphics[width=\linewidth]{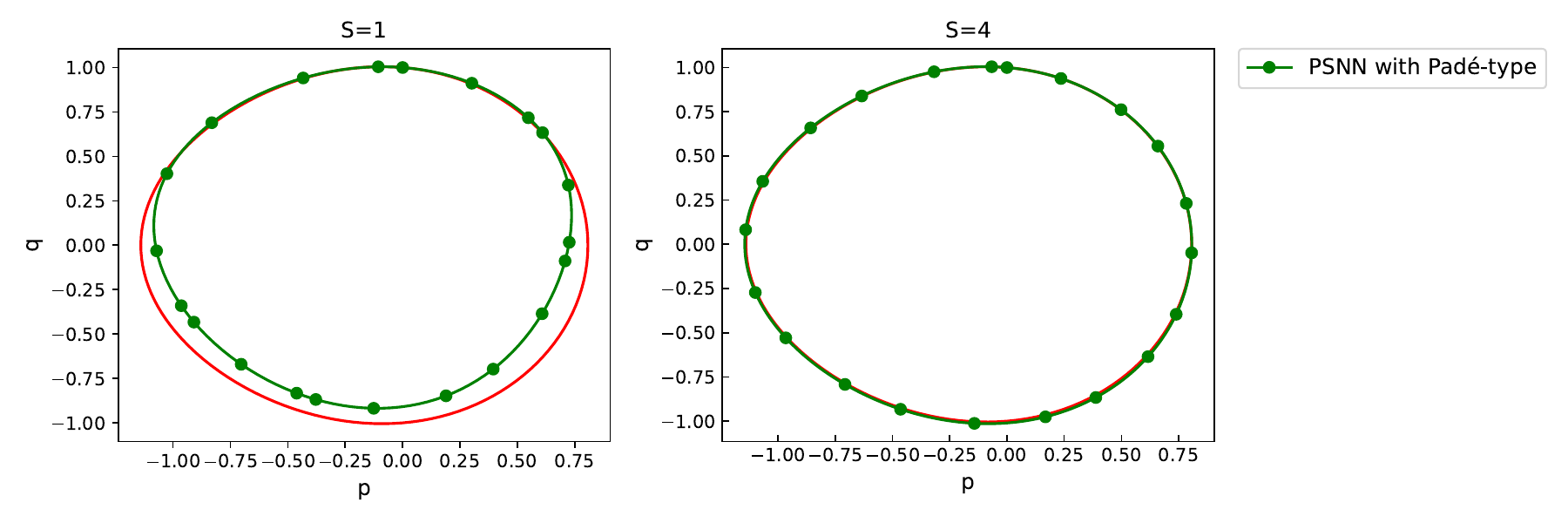}
\caption{Predicted phase trajectory on $t\in[0,60]$ by PSNN using Pad\'e-type activation function with $S=1$ (left); Predicted phase trajectory on $t\in[0,60]$ by PSNN using Pad\'e-type activation function with $S=4$ (right).}\label{2.1}
\end{figure}

Figure \ref{2.1} shows the impact 
of different number $S$ of Pad\'e-type summands in $H_{\boldsymbol{y},net}$ \eqref{hynetform} on the prediction ability of the PSNN. As can be seen from the figure, $S=4$ already suffices to produce fairly accurate prediction, in contrast to the Taylor-nets where $S=8$ is needed for a satisfactory prediction (\cite{h-6}).

\subsection{Example 3: A 4-dimensional galactic dynamic problem}
In this section we consider the learning of a 4-dimensional galactic dynamic problem (\cite{Hairer2}) which is described by a Hamiltonian system with Hamiltonian function
\begin{align}\label{hh-3}
H(p,q)=\frac{1}{2}(p_1^2+p_2^2)+(p_1q_2-p_2q_1)/2+ln(1+q_1^2+q_2^2).
\end{align}

Due to the higher dimensionality of the system, we opted for a larger amount of training and testing data. Denote $\boldsymbol{y}_j^{(i)}=(p_{1,j}^{(i)},p_{2,j}^{(i)},q_{1,j}^{(i)},q_{2,j}^{(i)})$ for $j=0,1$. Then the training dataset is $\mathcal{D}_{train}=\{(\boldsymbol{y}_0^{(i)},\boldsymbol{y}_1^{(i)})_{train}|i=1,\cdots,N_{train}\}$, where we set $N_{train}=1000$ and $\boldsymbol{y}_0^{(i)}$ are sampled uniformly from the region $V=[-2, 2] \times [-2, 2]\times [-2, 2]\times [-2, 2]$. The testing dataset is $\mathcal{D}_{test}=\{(\boldsymbol{y}_0^{(i)},\boldsymbol{y}_1^{(i)})_{test}|i=1,\cdots,N_{test}\}$ where $N_{test}=2000$ and $\boldsymbol{y}_0^{(i)}$ are sampled uniformly from $V$ and independently to that for $\mathcal{D}_{train}$. The $\boldsymbol{y}_{1}^{i}$ are generated by the midpoint rule with time step $h=0.01$ in both the training and testing data set. That is we take $T=0.01$ (the observation interval) and $K=T/h=1$ (iteration number of $\Phi$), similar to Example 1. As mentioned above we choose $S=6$ in the $H_{\boldsymbol{y},net}$ \eqref{hynetform} in order to better approximate the Hamiltonian function with higher-dimensional independent variable. Other parameters in the neural network settings are the same with those for Example 1.
\begin{figure}[ht]
\centering
\includegraphics[width=\linewidth]{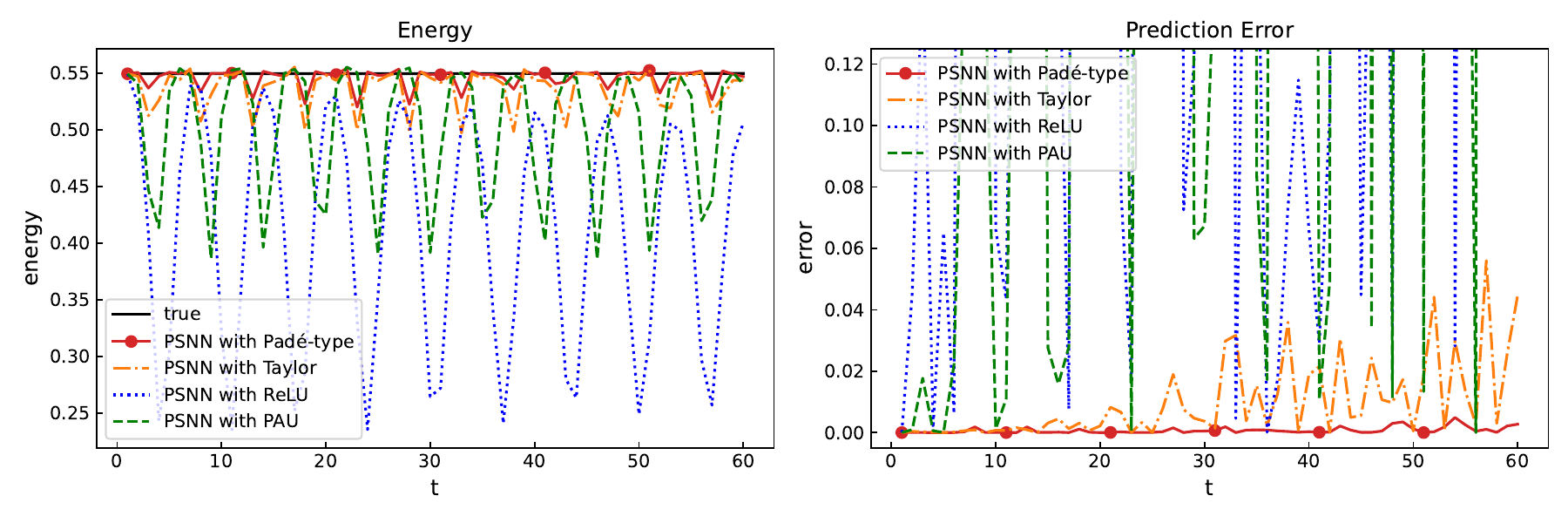}
\caption{Evolution of energy learned by PSNN with different activation functions (left); Prediction error by PSNN with different activation functions (right).}\label{3.3}
\end{figure}

The left panel of Figure \ref{3.3} shows the evolution of the Hamiltonian function \eqref{hh-3} on $t\in[0,60]$ learned by PSNN with different activation functions, where we can see that the Pad\'e-type activation function creates the best energy-preservation while the ReLU violates the energy conservation law mostly among the four. The right panel illustrates the prediction error \eqref{errorp-1} on $t\in[0,60]$ by applying PSNN with different activation functions. Again the Pad\'e-type activation function is significantly superior to the other three in prediction accuracy.

\begin{figure}[ht]
\centering
\includegraphics[width=\linewidth]{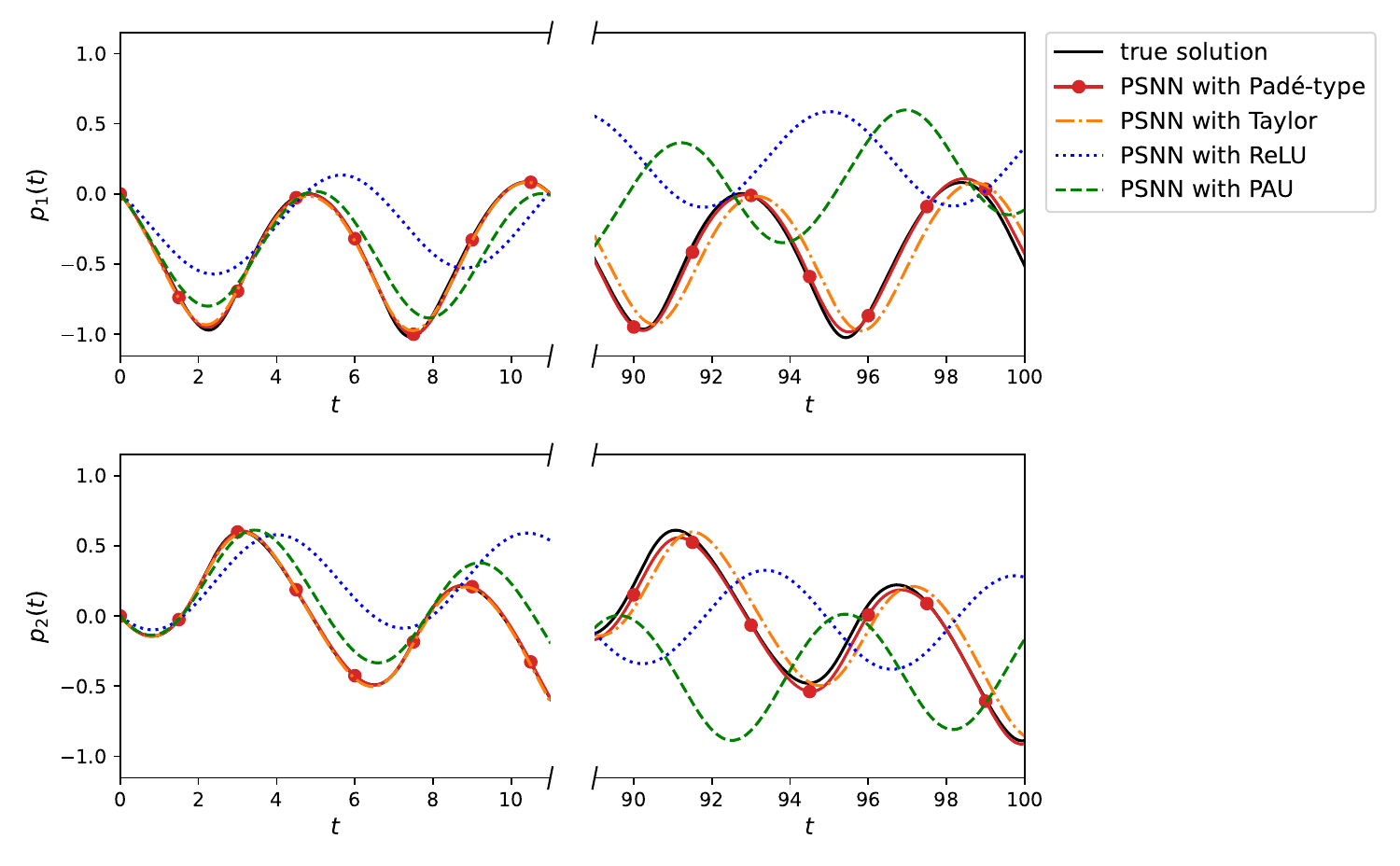}
\caption{Predicted $p_1(t)$ and $p_2(t)$ on $t\in[0,100]$ by PSNN with different activation functions.}\label{3.5}
\end{figure}
Figure \ref{3.5} contributes to compare the long-term prediction of $p_i(t)$ $(i=1,2)$ by using PSNN with different activation functions, where the Pad\'e-type activation function provides predicted solutions that almost coincide with the true solution even after a long time (10000 time steps) evolution. This  shows the high accuracy and long-term robustness of PSNN with Pad\'e-type activation function. 

\begin{figure}[ht]
\centering
\includegraphics[width=\linewidth]{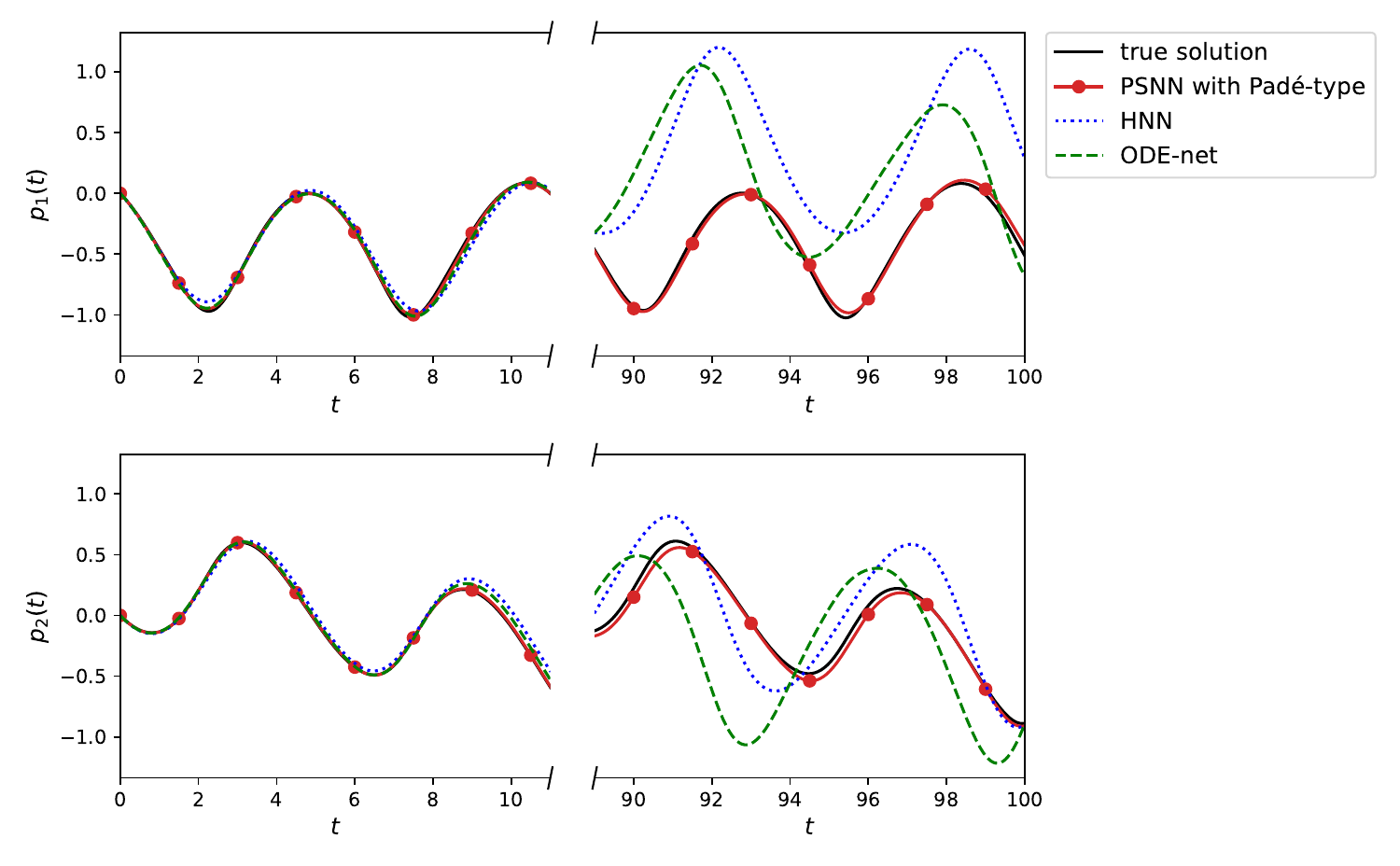}
\caption{Predicted $p_1(t)$ and $p_2(t)$ generated by PSNN, HNN and ODE-net, where PSNN uses the Pad\'e-type activation function.}\label{3.6}
\end{figure}
Figure \ref{3.6} presents a comparison between the PSNN and other models in predicting $p_i(t)$ ($i=1,2$) over the long time interval $t\in[0,100]$. It is obvious that the PSNN has significant advantages in terms of prediction accuracy and long-term robustness, in handling the higher-dimensional problem.

\begin{figure}[ht]
\centering
\includegraphics[width=\linewidth]{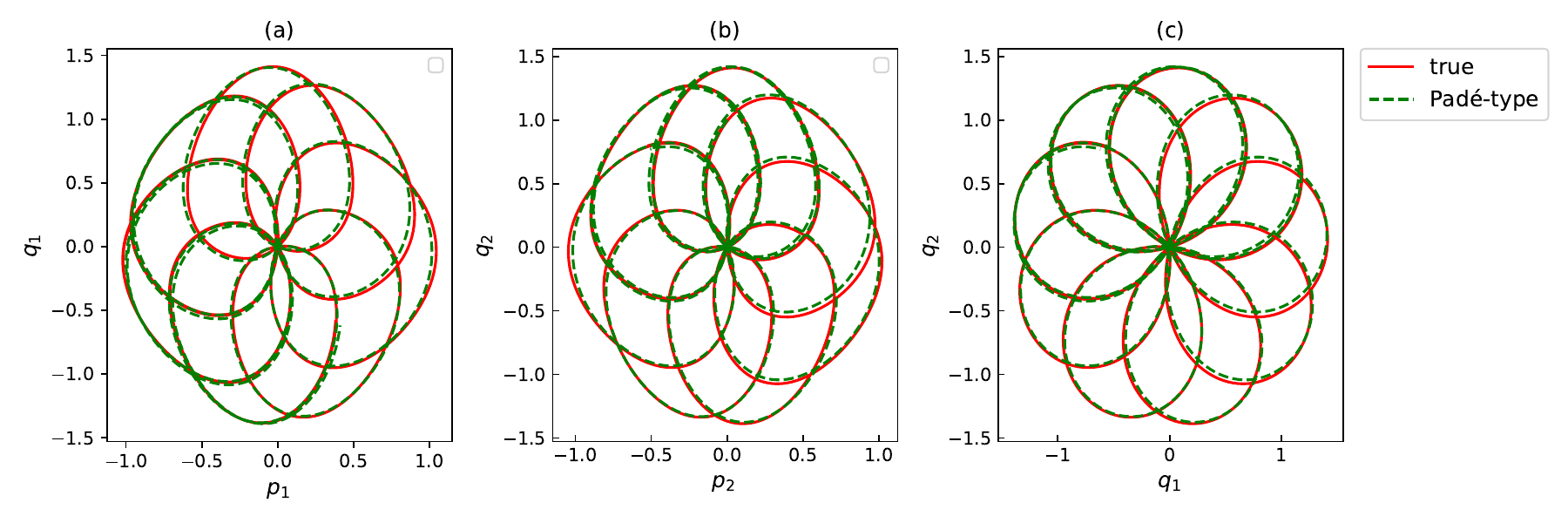}
\caption{Prediction of PSNN with Pad\'e-type activation function over $t\in[0,100]$ (a) on $(p_1, q_1)$ plane; (b) on $(p_2, q_2)$ plane; (c) on $(q_1, q_2)$ plane.}\label{3.1}
\end{figure}
Figure \ref{3.1} compares the prediction of PSNN using Pad\'e-type activation function with the true solution on $(p_1, q_1)$, $(p_2, q_2)$, and $(q_1, q_2)$ plane, respectively. Good coincidence between the predicted and true curves can be seen over the long time interval $t\in[0,100]$.

\FloatBarrier
\subsection{Example 4: A separable Hamiltonian system}
To compare our PSNN model with the Taylor-nets proposed in \cite{h-6} which are designed to learn separable Hamiltonian systems, we consider here the 2-dimensional pendulum system (\cite{h-6}), where the Hamiltonian function has the separable form
\begin{align}\label{h-1-old}
H(p,q)=\frac{p^2}{2}-\cos(q).
\end{align}

For learning the system \eqref{h-1-old}, the training and testing data are collected in the same way as in Example 1, and the network settings for PSNN also follows those in Example 1. The Taylor-net is trained according to the code and parameters provided in \cite{h-6}. 

\begin{figure}[ht]
\centering
\includegraphics[width=\linewidth]{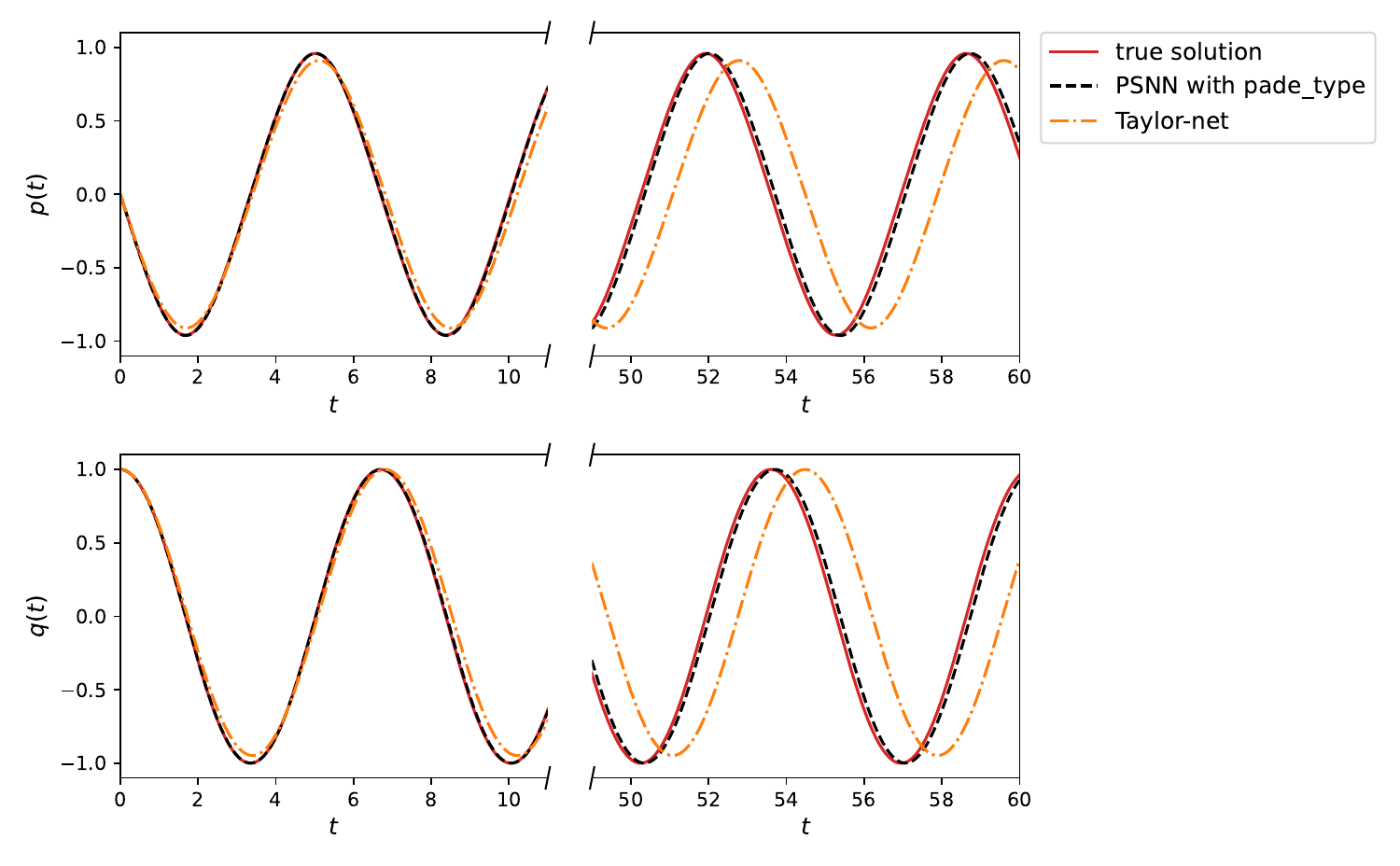}
\caption{Predicted solutions generated by PSNN and Taylor-net, where PSNN uses the Pad\'e-type activation function.}\label{0.5}
\end{figure}

Figure \ref{0.5} demonstrates the prediction of $p(t)$ and $q(t)$ of system \eqref{h-1-old} by using PSNN with Pad\'e-type activation function and the Taylor-net, respectively. As can be seen the PSNN shows better long-term prediction capability than the Taylor-net, even in learning separable Hamiltonian system. 
\begin{table}[h]
    \centering
    \begin{tabular}{ccccc}
    \hline
     Network&Training samples&S&Parameters&Error\\
    \hline
    PSNN & 15&4&274 &\textbf{0.1799}\\
           Taylor-net & 15&8& 514&0.8839\\
  \hline
    \end{tabular}
    \caption{A comparison between PSNN and Taylor-net in learning the separable Hamiltonian system \eqref{h-1-old}.}
    \label{tab0}
\end{table}

Table \ref{tab0} compares some neural network settings and corresponding prediction errors by using the PSNN and Taylor-net in the training, where  
the `Training samples' refers to the sample size of the training data, `S' represents the number of activation summands used in (\ref{8})-(\ref{9}) or (\ref{hynetform}), `Parameters' represents the total number of learnable parameters in the neural network, and `Error' denotes the prediction error of the neural network which is defined as
\begin{align}\label{errorr-2}
Error:= \frac{1}{60000}\sum_{i=1}^{60000}\frac{1}{2d}\|\boldsymbol{y}_i-\Phi^i(\boldsymbol{y}_0,h,H_{\boldsymbol{y},net})\|_2^2,
\end{align}
where $\{\boldsymbol{y}_i\}_{i=1}^{60000}$ represents a discrete trajectory of the true solution which are simulated by using a midpoint method with time step $0.01$, spanning from $t=0$ to $t=600$, and $2d$ is the dimension of the vectors $\boldsymbol{y}_i$. It can be seen that the PSNN has used less activation summands and network parameters to produce more accurate prediction than the Taylor-net, even for separable Hamiltonian system.

\section{Conclusion}\label{sec5}
The numerical experiments above demonstrate the superiority of our PSNN model and the Pad\'e-type activation function over traditional or state-of-the-art models and activation functions in learning Hamiltonian systems. In addition to the previous graphical comparison, we also summarized some experimental data across Examples 1 to 3 on non-separable Hamiltonian systems, listed in the following Tables.
\begin{table}[h]
    \centering
    \begin{tabular}{c|c|c|c|c|c|c|c|c}
    \hline
    Problem& &Pad\'e-type&\multicolumn{2}{c|}{PAU}& \multicolumn{2}{c|}{Taylor}&\multicolumn{2}{c}{ReLU}\\
    \hline
    \multirow{3}{*}{\emph{Example 1}}&Training samples &15& 15& 100& 15&100&15& 1000 \\
           &S &4 &4&4& 8& 8& 8& 8\\
           &Prediction error & \textbf{0.0345}&3.624&0.4270 &0.5762& 0.3284& 0.7273& 0.1980\\
           \hline
    \multirow{3}{*}{\emph{Example 2}}&Training samples &15 &15 &100 &15&100& 15&1000\\
           &S &4 & 8&4&8& 8& 8& 8\\
           &Prediction error &\textbf{0.0742} &0.3624&0.0975 & 0.6802 &0.1031 &5.1816 &0.1162\\
           \hline
   \multirow{3}{*}{\emph{Example 3}}&Training samples& 1000& 1000& 5000& 1000&5000&1000&10000 \\
           &S &6 & 12& 6&12&12& 12&12 \\
           &Prediction error & \textbf{0.5605}& 5.3687&0.7182&7.9136&2.579 &1.9445 &0.9316 \\
  \hline
    \end{tabular}
    \caption{Comparison of Pad\'e-type, Taylor, PAU and ReLU activation functions across three examples.}
    \label{tab2}
\end{table}

 It can be observed from Table \ref{tab2} that the Pad\'e-type activation function make the neural network achieve the minimum prediction error with the least amount of training data and activation summands among the comparators, where the `Prediction error' is defined as \eqref{errorr-2}.

\begin{table}[h]
    \centering
    \begin{tabular}{cccc}
    \hline
    Problem& Network &Epoch& Parameters\\
    \hline 
    \multirow{3}{*}{\emph{Example 1}}&PSNN & 1500&274 \\
           &HNN  &10000 & 2000\\
           &ODE-net & 10000& 2000\\
           \hline
    \multirow{3}{*}{\emph{Example 2}}&PSNN &1500 & 274\\
           &HNN  &10000 & 2000\\
           &ODE-net &10000 & 2000\\
           \hline
   \multirow{3}{*}{\emph{Example 3}}&PSNN & 2000& 216\\
           &HNN  & 50000& 5000\\
           &ODE-net & 50000& 5000\\
  \hline
    \end{tabular}
    \caption{Comparison of PSNN, HNN, and ODE-net across three examples.}
    \label{tab3}
\end{table}

While achieving better prediction accuracy and robustness as was shown by the previous figures in Examples 1 to 3, the PSNN with Pad\'e-type activation function has also used much fewer training epochs and learnable network parameters than the HNN and ODE-net, as can be seen from Table \ref{tab3}.  

In summary, both the graphical and data comparison demonstrate that our PSNN model with Pad\'e-type activation function outperforms classical feedforward Hamiltonian learning strategies evidently. This is mainly owing to the pseudo-symplectic structure of the neural network that highly inherits the intrinsic geometric property of the latent dynamical system on one hand, and the Pad\'e-type activation function which retains the high performance of Pad\'e approximation on the other hand. 

Furthermore, the application of pseudo-symplectic methods enables explicit  transmission of signals in the neural network with negligible loss of symplecticity. This generalizes the structure-preserving end-to-end feedforward learning of Hamiltonian systems from separable to also non-separable non-linear realm. Although the Pad\'e-type activation function is only used in the PSNN to learn Hamiltonian systems in this paper, the role it can play in neural networks may be of great potential, such as via optimal choice of $L,M$ and $S$, and wider applications to various scenarios, linking approximation theory with the interpretability research of neural networks.




\section*{Acknowledgments}
The first, second and fourth authors are supported by the NNSFC No. 11971458.


\end{document}